\documentclass{lipics-v2021}
\nolinenumbers
\usepackage[OT2,OT1]{fontenc}

\newcommand{\cyr}{%
  \renewcommand{\rmdefault}{wncyr}%
  \renewcommand{\sfdefault}{wncyss}%
  \renewcommand{\encodingdefault}{OT2}%
  \normalfont\selectfont
}
\DeclareTextFontCommand{\textcyr}{\cyr}

\usepackage{amsmath,amssymb,amscd,amsfonts,amsthm}
\usepackage{latexsym}
\usepackage{esvect}

\usepackage{bussproofs}
\usepackage{natded}
\usepackage{proof}
\usepackage[tableaux]{prooftrees}
\usepackage{forest}

\newcommand{\fillBox}{\hfill$\blacktriangleleft$}

\usepackage{graphicx}
\graphicspath{{images/}}
\usepackage{float}
\usepackage{tikz}
\usetikzlibrary{matrix,positioning}

\usepackage{multicol}
\usepackage{ragged2e}
\usepackage{tabularx}
\usepackage{enumitem}
\usepackage{yfonts}
\usepackage{csquotes}
\usepackage{comment}
\usepackage{xspace}

\usepackage{xargs}
\usepackage{xparse}

\usepackage{color}

\usepackage{url}

\newcommand{\deriv}[1]{\vdash_{#1}}       
\newcommand{\supp}[1]{\Vdash_{#1}} 
\newcommand{\arenasupp}[1]{\Vdash_{#1}'}  
\newcommand{\wstrat}[1]{\blacktriangleright_{#1}}            

\newcommand{\base}[1]{\mathcal{#1}}                    
\newcommand{\arena}[1]{\mathbb{#1}}                     
\newcommand{\sarena}{\mathbb{N}}                       
\newcommand{\sbase}{\mathcal{N}}                       

\newcommand{\brule}[1]{\triangleright{#1}}         

\newcommand{\atree}[2]{#1 \Rightarrow #2}              

\newcommand{\tobase}[1]{\mathrel{\textsf{Base}}(#1)}            
\newcommand{\toarena}[1]{\mathrel{\textsf{Arena}}(#1)}           

\newcommand{\shrp}[1]{#1^{\natural}}                          
\newcommand{\flatset}[1]{#1^{\flat}}                    
\newcommand{\shrpset}[1]{#1^{\natural}}                       

\newcommand{\At}{\mathit{At}}                           

\newcommand{\NJ}{\mathrm{NJ}}
\newcommand{\IPL}{\textsf{IPL}}

\title{Inferentialist Game Semantics \\ (Extended Abstract)} 
\titlerunning{Inferentialist Game Semantics}

\author{Joaquim T. Waddington\footnote{Corresponding author.}}{University College London}
{joaquim.waddington.24@ucl.ac.uk}{}{}
\author{Alexander V. Gheorghiu}{University College London}
{alexander.v.gheorghiu@gmail.com}{}{}
\author{David J. Pym}{UCL \& Institute of Philosophy, University of London}
{david.pym@sas.ac.uk}{}{}
\authorrunning{Waddington, Gheorghiu, Pym} 

\Copyright{Joaquim T. Waddington, Alexander V. Gheorghiu, David J. Pym}



\relatedversion{} 

\acknowledgements{Waddington is supported by the Leverhulme Trust grant RPG-2024-196 and has received funding from the European Union’s Horizon 2020 research and innovation program under the Marie Skłodowska-Curie grant agreement Number 101299559. Pym is grateful to UCL for its support of his sabbatical at the University of London's Institute of Philosophy and to the Institute for its generous hosting. The authors are grateful to Pinaki Chakraborty and Elaine Pimentel for their interest and support.}

\begin{document}

\maketitle

\begin{abstract}
Game semantics is an elegant approach to the formal semantics of reasoning and computation that grounds model-theoretic concepts of truth and validity in game-theoretic concepts that emphasize the dynamic and interactive aspects of logical reasoning. In Hyland-Ong games, plays are traces of interactions between a player and an environment and such games provide a naturally appealing semantics for computation that is derived from proof-search in logical systems. Such a semantics can be seen as providing an intensional theory of meaning for systems of logic in terms of (the computation of) proofs. In logic, an intensional theory of meaning for systems of logic is offered by proof-theoretic semantics; in particular, by base-extension semantics (B-eS), in which the model-theoretic interpretation of atomic propositions in a satisfaction relation is replaced by a validity relation which uses provability in `bases' of atomic rules. We establish a fully abstract correlation between B-eS and Hyland-Ong game semantics, employing techniques similar to those used by Sandqvist to give a sound and complete B-eS for intuitionistic propositional logic. 
We illustrate our semantics through the example of 
4x4 Sudoku. 
\end{abstract}

\section{Introduction}
\label{sec:introduction}
 
 Game semantics is an elegant approach to the formal semantics of reasoning and computation that grounds model-theoretic concepts of truth and validity in game-theoretic concepts that emphasize the dynamic and interactive aspects of logical reasoning. A formula is interpreted by a two-player game between a \emph{Proponent}, who defends a claim, and an \emph{Opponent}, who represents the environment. The game specifies which challenges Opponent may issue and how Proponent may meet them; a proof, or a program, is then interpreted as a \emph{winning strategy} --- a policy by which Proponent meets every challenge that Opponent can mount. The dialogical analysis of logical validity in this style goes back to Lorenzen~\cite{Lorenzen1960,LorenzenLorenz1978}, for whom a formula is intuitionistically valid just in case Proponent has a winning strategy in the associated dialogue --- see Felscher~\cite{Felscher1985} for a careful technical development. In the semantics of computation, the dialogical idea was sharpened into a systematic approach in the work of Hyland and Ong~\cite{HYLAND2000285}. The methodology has been extended several times; see, for example, Blass~\cite{Blass1992} and Abramsky et al.~\cite{AbramskyJagadeesan1994, AJM2000}, and Nickau~\cite{Nickau1994}.
 
Two features of this picture matter for the present paper. First, game semantics is \emph{intensional}: a strategy records not merely \emph{that} a claim is defensible but \emph{how} it is defended, move by move. Second, the interaction it describes reads naturally as \emph{proof-search}: Opponent's challenges decompose a putative conclusion into sufficient sub-goals, and Proponent's responses discharge them. Indeed, Hyland--Ong games have been used to give a semantics for proof-search in reductive logic~\cite{Pym_ReductiveLogic} (game semantics for proof-search has also been considered in~\cite{MS2006}). Read this way, game semantics provides an intensional theory of meaning for systems of logic in terms of (the computation of) proofs. 
 
An alternative intensional theory of meaning has been developed under the title \emph{proof-theoretic semantics} (P-tS). In this framework, the meaning of the logical constants is explicated not by reference to truth in models but by reference to proof-theoretic constructions. P-tS may be seen as a logical realization of \emph{inferentialism} \cite{Brandom1994,Brandom2000,Brandom2024}, as developed by Dummett~\cite{Dummett1975,Dummett1991} and Prawitz~\cite{Prawitz1971ideas,Prawitz2006meaning}, in which the meaning of linguistic and logical constructs is determined by their use. P-tS in the sense we require here originates with Prawitz's `Ideas and Results in Proof Theory'~\cite{Prawitz1971ideas} (Appendix~A), in which Prawitz seeks a theory of validity for (natural deduction) \emph{proofs} that is grounded not in terms of denotations but in a pre-logical --- that is, connective-free --- notion of provability in a \emph{base} of atomic rules; see Schroeder-Heister~\cite{Schroeder2007modelvsproof} for the contrast with model-theoretic semantics. We defer the details to Section~\ref{sec:bes}. Concretely, the version of P-tS we adopt here is called \emph{base-extension semantics} and originates with Sandqvist~\cite{Sandqvist2005inferentialist,Sandqvist_Tor}.
 
As games and P-tS offer two intensional theories of meaning it is natural to ask if any correspondence can be made between them. That is, can the concepts of game semantics be reconstructed from the primitives of P-tS, so that the games are themselves proof-theoretic in character rather than an externally imposed model? 
In this paper, we answer this question positively for \emph{intuitionistic propositional logic} ($\IPL$). We reconstruct the basic concepts of Hyland--Ong game semantics from the primitives of B-eS according to the following dictionary: a base of atomic rules is an \emph{arena}; an argument from assumptions is a \emph{play}; a proof is a \emph{winning strategy}. Replacing provability-in-a-base by the existence of a winning strategy in the corresponding arena throughout Sandqvist's clauses yields a \emph{game-extension semantics} for $\IPL$: our main logical result is that $\IPL$ is sound and complete for it.  
 
Section~\ref{sec:bes} reviews Sandqvist's P-tS for $\IPL$ \cite{Sandqvist_Tor}, defining bases, derivability in a base, and the support relation; these definitions are used throughout the paper. Section~\ref{sec:games} develops the game-theoretic apparatus from the inferentialist point of view realized by P-tS: atomic arenas, moves, plays, and (winning) strategies, with the dictionary above made precise. The point of departure is the use of Hyland--Ong games in~\cite{Pym_ReductiveLogic} to give a semantics for proof-search in reductive logic. Section~\ref{sec:sudoku} gives a substantial worked example, namely $4 \times 4$ Sudoku, which shares the player-vs-environment setup of the game semantics for proof-search in~\cite{Pym_ReductiveLogic} while being a familiar game that is not immediately about logical consequence. Section~\ref{sec:arenas-and-bases} establishes the key properties of the basic setup: the correspondence between derivability in a base and winning strategies in the associated arena. Section~\ref{sec:metatheory} then delivers the main result: NJ is sound and complete for the validity relation of the game-extension semantics. Finally, Section~\ref{sec:discussion} summarizes the contribution and suggests directions for further research. Some proofs are deferred to the appendices.
 
\noindent \textbf{Notations \& Conventions.}
Lower-case italic letters $p, q, r, s$, etc.\ denote propositional atoms, and upper-case italic letters $P, Q, R, S$, etc.\ denote finite (possibly empty) sets thereof. It is essential that we have a denumerable set of atoms $\At = \{p_1, p_2, \ldots\}$.
For formulae, we use lower-case Greek letters $\phi, \psi, \gamma$, etc., and for finite (possibly empty) sets of formulae upper-case Greek letters $\Gamma, \Delta, \Sigma$, and so on. We write $\base{B}, \base{C}, \base{D}$, etc.\ for bases, $\deriv{\base{B}}$ for
derivability in a base $\base{B}$, and $\supp{\base{B}}$ for the support relation relative to a base $\base{B}$. For arenas, we use $\arena{A}, \arena{B, \arena{C}}$ etc., and $\arenasupp{\arena{A}}$ for the support relation relative to an arena $\arena{A}$. We write $\vdash_{\NJ}$ to denote provability/consequence
in the natural deduction system NJ \cite{Gentzen1935,Prawitz1965} (see Figure~\ref{fig:nj} in Appendix~\ref{app:metatheory}).
 
\section{Background: Proof-theoretic Semantics}
\label{sec:bes}

The version of P-tS we use in this paper is the \emph{base-extension semantics} (B-eS) for $\IPL$ given by Sandqvist~\cite{Sandqvist_Tor}.
 This section defines the semantics in full.

The primitive of  B-eS is the idea of a \emph{base} --- that is, a set of rules over atomic formulae (i.e., propositional atoms, not $\bot$). They function as inferential models in caparison to the denotational models of Tarski/Beth/Kripke semantics. No logical constants occurs anywhere in a base; this is what makes derivability in a base a pre-logical notion, fit to ground a semantics for the connectives without circularity. The logical constants are defined inductively by a semantic judgment called \emph{support}.

 
\begin{definition}[Base rule]
\label{def:base-rule}
A \emph{$0$th-level base rule} is a singleton $\langle r \rangle$, where $r$ is an
atom. A \emph{$1$st-level base rule} is an ordered pair $\langle Q, r \rangle$,
where $r$ is an atom and $Q$ is a finite (possibly empty) set of atoms. An
\emph{$(n+1)$th-level base rule} is an ordered pair $\langle Q, r \rangle$, where
$r$ is an atom and $Q$ is a finite set of $k$th-level base rules with $k \leq n$. \fillBox
\end{definition}
 
We identify $\langle \emptyset, r \rangle$ with $\langle r \rangle$. For base rules we write
\[
(P_1 \brule q_1), \dots, (P_n \brule q_n) \brule r
\]
in place of $\langle \{ \langle P_1, q_1 \rangle, \ldots, \langle P_n, q_n \rangle \}, r \rangle$.
Such a rule is meant to be read as: \emph{given derivations of $q_1$ through $q_n$, discharge the premisses $P_1$ through $P_n$ and conclude $r$}.
\begin{definition}[Base]
\label{def:base}
A \emph{base} $\base{B}$ is a (possibly infinite, possibly empty) set of base rules. A base $\base{C}$ \emph{extends} $\base{B}$, written $\base{C} \supseteq \base{B}$, iff every rule of $\base{B}$ is a rule of $\base{C}$. \fillBox
\end{definition}
\begin{definition}[Derivability in a base]
\label{def:derivability-in-a-base}
Let $\base{B}$ be a base. The \emph{derivability relation} $\deriv{\base{B}}$,
relating sets of atoms to atoms, is defined
inductively by the following clauses:
\[
\begin{array}{rl}
    \mbox{\rm (Ref)} & \mbox{if $p \in S$, then $S \deriv{\base{B}} p$} \\
    \mbox{\rm (App)} & \mbox{if $(Q_1 \brule q_1), \ldots, (Q_n \brule q_n)
        \brule r$ is a rule of $\base{B}$ and
        $S \deriv{\base{B} \cup Q_i} q_i$, for each $i = 1,...,n$,} \\
        & \mbox{then $S \deriv{\base{B}} r$}
\end{array}
\]
For this to make sense, we identify an atom $p$ with the
$0$th-level rule $\langle p \rangle$. \fillBox
\end{definition}
\begin{figure}[t]
    \hrule
    \begin{align}
        &\supp{\base{B}} p
            &&\text{iff} \qquad\deriv{\base{B}} p
            \tag{At} \\
        &\supp{\base{B}} \phi \to \psi
            &&\text{iff} \qquad\phi \supp{\base{B}} \psi
            \tag{$\to$} \\
        &\supp{\base{B}} \phi \land \psi
            &&\text{iff} \qquad\supp{\base{B}} \phi \text{ and } \supp{\base{B}} \psi
            \tag{$\land$} \\
        &\supp{\base{B}} \bot
            &&\text{iff} \qquad\supp{\base{B}} p \text{, for every atom } p
            \tag{$\bot$} \\
        &\supp{\base{B}} \phi \lor \psi
            &&\text{iff} \qquad\text{for all } \base{C} \supseteq \base{B} \text{ and every atom } p, \notag \\
            & && \qquad \qquad\text{if } \phi \supp{\base{C}} p \text{ and } \psi \supp{\base{C}} p \text{, then } \supp{\base{C}} p
            \tag{$\lor$} \\
        &\Delta \supp{\base{B}} \phi
            &&\text{iff} \qquad \text{for all } \base{C} \supseteq \base{B},\ \text{if } \supp{\base{C}} \psi \text{, for all } \psi \in \Delta \text{, then } \supp{\base{C}} \phi
            \tag{Inf}
    \end{align}
     \caption{Sandqvist's B-eS for $\IPL$ validity \cite{Sandqvist_Tor}}\vspace{1mm}
    \hrule
    \label{fig:sandqvist-ipl}
\end{figure}
\begin{remark}[Restriction]
\label{rem:restriction}
Assuming an indefinite supply of fresh atoms, bases comprising at most
$2$nd-level rules are as expressive as bases of arbitrary level. That is,
following Stafford et al.~\cite{Stafford2026-STALOP-5}: for any base $\base{B}$, there is an at
most $2$nd-level base $\base{B}'$ such that
$\deriv{\base{B}} p$ iff  $\deriv{\base{B}'} p$,
for every atom $p$ occurring in $\base{B}$. \fillBox
\end{remark}

\begin{definition}[Support in a Base]
\label{def:support}
Let $\base{B}$ be a base. The \emph{support} relation $\supp{\base{B}}$,
relating finite (possibly empty) sets of formulae to formulae, is defined
inductively by the clauses of Figure~\ref{fig:sandqvist-ipl}, in which the base
case (At) is the derivability relation of
Definition~\ref{def:derivability-in-a-base}, and in (Inf) $\Delta \neq \emptyset$. We say that $\Gamma$
\emph{supports} $\phi$, written $\Gamma \Vdash \phi$, iff
$\Gamma \supp{\base{B}} \phi$ for every base $\base{B}$. \fillBox
\end{definition}

At first glance, the semantics seems to follows  Kripke/Beth semantics in a proof-theoretic key. This is misleading: note the form of the clause for disjunction. The efficacy of this choice is discussed extensively in~\cite{Sandqvist_Tor} and elsewhere (e.g.~\cite{GPRR-ContCompPtS2026}). The salient point is that using the definition \,
\[
\supp{\base{B}} \phi \lor \psi \qquad \mbox{iff} \qquad \supp{\base{B}} \phi \text{ or } \supp{\base{B}} \psi
\]
does \emph{not} give $\IPL$~\cite{Piecha2015failure,Piecha2016completeness,Piecha2019incompleteness,Stafford2026-STALOP-5}. 
 
\begin{theorem}[Sandqvist~\cite{Sandqvist_Tor}]
\label{thm:sandqvist-soundness-completeness}
Let $\Gamma$ be finite. Then $\Gamma \vdash_{\NJ} \phi$ iff $\Gamma \Vdash \phi$. \fillBox
\end{theorem}

As Sandqvist~\cite{Sandqvist_Tor} uses only second-level bases, this follows from  Remark~\ref{rem:restriction}.  Sandqvist~\cite{Sandqvist2005inferentialist} and Makinson~\cite{makinson2014inferential} have considered also B-eS for classical logic.  Gheorghiu \cite{Gheorghiu2025-FirstOrder} has developed B-eS uniformly for classical and intuitionistic first-order predicate logic, and Gheorghiu and Pym have treated second-order logic~\cite{GP2025-P-tS-S-oL}. There is by now a substantial body of work, too extensive to survey here, on B-eS for modal and substructural logics (e.g., \cite{BuzokuPym-Modal2026,EckhardtPym2025Modal,EckhardtPym2025S5,GheorghiuGuPym2023IMLL,BI-P-tS2025,buzoku2026ILL}).

\section{Game-theoretic Semantics}
\label{sec:games}
Following Hyland and Ong~\cite{HYLAND2000285}, the picture to have in mind is of a dialogue between two players, Proponent and Opponent. A play is a sequence of moves conducted within the discipline of an \emph{arena}, which determines the legal actions. Different choices of arena and of the rules of play determine different games. Pym and Ritter~\cite{Pym_ReductiveLogic} gave a version that they showed to be identical to proof-search in intuitionistic logic: a winning strategy --- a function determining the moves Proponent should play in order to guarantee a win --- corresponds to a proof, and the sequence of moves it generates corresponds to the construction of that proof.
 
The games of this section take the Pym--Ritter games as
their point of departure, but the setup cannot simply be inherited. Our target is not merely the construction of proofs, but the proof-theoretic semantics of the previous section. Accordingly, the apparatus must be modified, and this is delicate: slightly different setups may result in quite different games. The definitions that follow effect this calibration;
that they succeed is the content of Section~\ref{sec:arenas-and-bases}.
\begin{definition}[Atomic arena]
\label{def:atomic-arena}
An \emph{atomic arena} is a forest --- that is, a set --- of rooted trees of atoms. We use $\arena{A}, \arena{B}, \arena{C}, \ldots$ to denote arenas. \fillBox
\end{definition}
 
Let $\arena{A}$ and $\arena{B}$ be arenas. We say that $\arena{B}$ \emph{extends} $\arena{A}$, written $\arena{B} \supseteq \arena{A}$, iff the set of trees of $\arena{A}$ is a subset of the set of trees of $\arena{B}$.
 
In an arena $\arena{A}$, the root nodes are \emph{P-nodes}, their children are \emph{O-nodes}, and the children of O-nodes are again P-nodes. The polarity of nodes records which player is responsible for the corresponding claims: root nodes are conclusions that Proponent may seek to establish, their children are premisses that Opponent may challenge, and so on, alternating with depth.
 
We turn to the dynamic part of the dictionary: moves and plays. As in Hyland--Ong games~\cite{HYLAND2000285}, moves are of two kinds --- questions, which demand justification, and answers, which supply or concede it --- and each player may both question and answer, according to the polarity of the node at which the move is made.
\begin{definition}[Move]
\label{def:move}
A \emph{move} $m$ is an action performed by one of the players together with a
datum, which is either an atom or a node of the arena. Each move is either a
\emph{question} or an \emph{answer}. We call a move by Proponent a
\emph{P-move} and a move by Opponent an \emph{O-move}. \fillBox
\end{definition}
 
The language of proof-search supplies useful intuitions for moves. The trees
of an arena $\arena{A}$ are the rules available to Proponent for answering
Opponent's questions. A question by Opponent is a demand for proof: we write
O?$p$ to denote Opponent asking `What is your justification for $p$?'. A
question by Proponent, correspondingly, invokes a rule: P?$q$ --- where $q$ is
a node of $\arena{A}$ --- asks `May I use the rule that concludes $q$?'.
Eventually, Opponent has no questions available and must answer instead:
O!$r$, where $r$ is an atom corresponding to an earlier question by Proponent,
denotes Opponent conceding `Very well, I grant $r$'. Proponent may use such
concessions to answer Opponent's questions: P!$s$, where $s$ is an earlier
question by Opponent, denotes Proponent declaring `So $s$ is justified'. This
intuition drives the following definition:
\begin{definition}[Play]
\label{def:play}
A \emph{play} for an arena $\arena{A}$ is a finite sequence of moves $m_0, m_1, \ldots, m_n$ satisfying the following conditions.
    \begin{enumerate}[label=(\roman*)]
        \item The move $m_0$ is an O-question, called the \emph{initial question}.
        \item There is an index $I \ge 0$ such that $m_1, \ldots, m_I$ are all O-answers,
      called the \emph{initial assumptions} (if $I = 0$ there are none).
        \item After the initial question and the initial assumptions, moves alternate between Opponent and Proponent.
        \item Every root question --- that is, every question whose node is a root of the
      forest of $\arena{A}$ --- is a P-question, and its \emph{justifying question}
      is the initial question.
\item For each question $m_j$ with $j > I$ that is not a root question, there is a
      question $m_k$ with $k < j$ whose node is the immediate predecessor, in
      $\arena{A}$, of the node of $m_j$; we call $m_k$ the \emph{justifying
      question} for $m_j$.
        \item For each answer $m_i$ with $i > I$, there is a question $m_k$ with $k < i$ at the same node of $\arena{A}$. If $m_j$ is the justifying question for $m_k$, we call $m_j$ the justifying question for $m_i$.
        \item Each question is answered at most once.
        \item The initial question is answered only once every non-initial question has been answered.
        \item For each P-answer $m_i$ there is an O-answer $m_j$ with the same label and $j < i$.
        \item The initial question is answered only if there is an O-answer to a root vertex with the same label as the initial question, or an initial assumption with the same label as the initial question.
    \end{enumerate}
    The datum of the initial question and of the initial assumptions is an atom; the datum of every other move is a node of the arena. The \emph{label} of a move is
its atom, or the atom labelling its node. \fillBox
\end{definition}
 
A strategy is Proponent's policy for continuing a play: it dictates Proponent's responses to whatever Opponent has done so far, and it must respond whenever the rules of play permit Opponent nothing further to say.
\begin{definition}[Strategy]
\label{def:strategy}
A \emph{strategy} for an arena $\arena{A}$ is a function $\sigma$ mapping
each play $\pi = (m_0,\ldots,m_k)$ whose last move is an O-move to
a (possibly empty) sequence of moves $\sigma(\pi) = (m_{k+1}, \ldots, m_n)$,
subject to the following conditions:
\begin{itemize}[align=left]
    \item \emph{(coherence)} the concatenation
          $\pi \cdot \sigma(\pi)$ is a play;
    \item \emph{(responsiveness)} if $\pi$ contains no unanswered P-questions that could be answered by Opponent in the next move, then $\sigma(\pi)$ is non-empty.
\end{itemize}
A strategy $\sigma$ is \emph{winning} for an atom $r$ from a set of
assumptions $Q$ if it satisfies the following further condition:
\begin{itemize}
    \item \emph{(winning)} every maximal play that opens with the initial
          question O?$r$ followed by the initial assumptions O!$q$, for
          $q \in Q$, and in which Proponent plays according to $\sigma$,
          contains a P-answer to the initial question.
\end{itemize}
We write $Q \wstrat{\arena{A}} r$ to express that there is a winning
strategy for $r$ from $Q$ in $\arena{A}$. \fillBox
\end{definition}
 
To conclude this section, we observe the following whose proof is deferred to Appendix~\ref{app:games}.
\begin{lemma}[Monotonicity]
\label{lem:monotonicity}
    If $Q \wstrat{\arena{A}} r$ and $\arena{B} \supseteq \arena{A}$, then $Q \wstrat{\arena{B}} r$. \fillBox
\end{lemma}

In this paper, we show that our notion of game exactly captures derivability in a base, the foundational concept of Sandqvist's B-eS. This characterization then allows us to explicate the defining clauses of the semantics, including the clause for disjunction. Before proceeding, however, we pause to develop intuitions for these definitions through an extended example that is not immediately about logical consequence: $4 \times 4$ Sudoku.
\section{Example: 4 x 4 Sudoku}
\label{sec:sudoku}
\NewDocumentCommand{\ncarena}{m}{
  \begin{tabular}{|c|c|c|c|} \hline
    #1 \\ \hline
  \end{tabular}}
 
As discussed, the concept of game developed here is modelled on proof-search.
Despite the dialogical presentation, therefore, the better analogy is with
single-player games of complete information: not chess or Go, but
solitaire or the fourteen-fifteen puzzle. The Opponent pursues no agenda of its own but
stands for the environment, granting or withholding whatever the rules
permit. In this section, we formalise $4 \times 4$ Sudoku.
 
The rules of $4 \times 4$ Sudoku are those of regular Sudoku~\cite{FelgenhauerJarvis2006}, but with a reduced board. A
\emph{state} of play is a board in which some cells, possibly none, contain
numbers from $1$ to $4$. We represent states as propositions: let
$\mathit{States} \subseteq \At$ be the set of all possible $4 \times 4$ Sudoku
states, legal and illegal alike. From
these atoms we build an arena --- which we call the \emph{Sudoku arena},
denoted $\arena{S}$ --- that models the relation between states.
 
In 
$\arena{S}$, trees have at most depth two. Each root vertex has at most one child. The root of each tree is labelled with a legal game state. The child of the root vertex is labelled with one of the possible game states accessible from the label of the root vertex. If the root vertex accesses no other game state, it has no child. Given a game state $p$ with accessible game states $q_1, \ldots, q_n$ ($n \geq 1$), there will be $n$ trees with the root vertex labelled $p$:
\begin{center}
    \begin{forest}
        [$p$[$q_1$]]
    \end{forest}
    \begin{forest}
        [$p$[$q_2$]]
    \end{forest}
    \raisebox{7mm}{\ldots}
    \begin{forest}
        [$p$[$q_n$]]
    \end{forest}
\end{center}
If $p$ is a legal game state with no accessible game states from it, then $\arena{S}$ contains a single-vertex tree labelled $p$. For example, consider the game state $p$:
    \[{\small
    \hspace{5.5cm}\ncarena{1 & 2 & 3 & 4 \\ \hline 3 & 4 & 1 & 2 \\ \hline 4 & 1 & 2 & 3 \\ \hline 2 & 3 & 4 & }}
    \]
There will be four trees with the root labelled $p$:
\[
{\small
\begin{array}{rc @{\qquad} rc @{\qquad} rc @{\qquad} rc}
p   & \ncarena{1 & 2 & 3 & 4 \\ \hline 3 & 4 & 1 & 2 \\ \hline 4 & 1 & 2 & 3 \\ \hline 2 & 3 & 4 & } & p & \ncarena{1 & 2 & 3 & 4 \\ \hline 3 & 4 & 1 & 2 \\ \hline 4 & 1 & 2 & 3 \\ \hline 2 & 3 & 4 & } & p & \ncarena{1 & 2 & 3 & 4 \\ \hline 3 & 4 & 1 & 2 \\ \hline 4 & 1 & 2 & 3 \\ \hline 2 & 3 & 4 & } & p & \ncarena{1 & 2 & 3 & 4 \\ \hline 3 & 4 & 1 & 2 \\ \hline 4 & 1 & 2 & 3 \\ \hline 2 & 3 & 4 & } \\
   &      & &      & &      & &      \\
   & \mid & & \mid & & \mid & & \mid \\
   &      & &      & &      & &      \\
q_1 & \ncarena{1 & 2 & 3 & 4 \\ \hline 3 & 4 & 1 & 2 \\ \hline 4 & 1 & 2 & 3 \\ \hline 2 & 3 & 4 & 1} & q_2 & \ncarena{1 & 2 & 3 & 4 \\ \hline 3 & 4 & 1 & 2 \\ \hline 4 & 1 & 2 & 3 \\ \hline 2 & 3 & 4 & 2} & q_3 & \ncarena{1 & 2 & 3 & 4 \\ \hline 3 & 4 & 1 & 2 \\ \hline 4 & 1 & 2 & 3 \\ \hline 2 & 3 & 4 & 3} & q_4 & \ncarena{1 & 2 & 3 & 4 \\ \hline 3 & 4 & 1 & 2 \\ \hline 4 & 1 & 2 & 3 \\ \hline 2 & 3 & 4 & 4} \\
\end{array}}
\]
 
Note that among the child vertices of $p$, only $q_1$ is a legal game state. Therefore, there will be no tree with a root vertex labelled $q_2$, $q_3$, or $q_4$. Since $q_1$ is an ending game state, there will be a tree with just one node labelled $q_1$. Therefore, Proponent has a winning strategy only if she chooses the tree $\atree{q_1}{p}$:
\begin{center}
    O?$p$, P?$p$, O?$q_1$, P?$q_1$, O!$q_1$, P!$q_1$, O!$p$, P!$p$
\end{center}
 
We consider another game, in which Opponent asks for:
    \[ {\small
    \hspace{5.5cm}\ncarena{1 & 2 & 3 & 4 \\ \hline 3 & 4 & 1 & 2 \\ \hline 4 & 1 & 2 & 3 \\ \hline  &  & 4 & } }
    \]
There is a winning strategy for Proponent:
\[
{\small
\begin{array}{rc @{\qquad} rc @{\qquad} rc @{\qquad} rc}
p   & \ncarena{1 & 2 & 3 & 4 \\ \hline 3 & 4 & 1 & 2 \\ \hline 4 & 1 & 2 & 3 \\ \hline  &  & 4 & } & q & \ncarena{1 & 2 & 3 & 4 \\ \hline 3 & 4 & 1 & 2 \\ \hline 4 & 1 & 2 & 3 \\ \hline 2 &  & 4 & } & r & \ncarena{1 & 2 & 3 & 4 \\ \hline 3 & 4 & 1 & 2 \\ \hline 4 & 1 & 2 & 3 \\ \hline 2 & 3 & 4 & } & s & \ncarena{1 & 2 & 3 & 4 \\ \hline 3 & 4 & 1 & 2 \\ \hline 4 & 1 & 2 & 3 \\ \hline 2 & 3 & 4 & 1} \\
   &      & &      & &      & &      \\
   & \mid & & \mid & & \mid & &      \\
   &      & &      & &      & &      \\
q & \ncarena{1 & 2 & 3 & 4 \\ \hline 3 & 4 & 1 & 2 \\ \hline 4 & 1 & 2 & 3 \\ \hline 2 &  & 4 & } & r & \ncarena{1 & 2 & 3 & 4 \\ \hline 3 & 4 & 1 & 2 \\ \hline 4 & 1 & 2 & 3 \\ \hline 2 & 3 & 4 & } & s & \ncarena{1 & 2 & 3 & 4 \\ \hline 3 & 4 & 1 & 2 \\ \hline 4 & 1 & 2 & 3 \\ \hline 2 & 3 & 4 & 1} &  &  \\
\end{array}}
\]
\begin{center}
    O?$p$, P?$p$, O?$q$, P?$q$, O?$r$, P?$r$, O?$s$, P?$s$, O!$s$, P!$s$, O!$r$, P!$r$, O!$q$, P!$q$, O!$p$, P!$p$
\end{center}
\section{Faithfulness and Adequacy}
\label{sec:arenas-and-bases}
This section establishes
that the game-theoretic apparatus of
Section~\ref{sec:games} is an exact counterpart of the proof-theoretic
apparatus of Section~\ref{sec:bes}. Two grades of correspondence must be
distinguished:
\begin{itemize}
    \item \emph{Adequacy}. This is extensional equivalence: the game interpretation is
adequate for derivability in a base when an atom is derivable from given
assumptions iff Proponent has a winning strategy for it in the corresponding
arena. It settles the yes/no question of derivability and nothing
more.
\item \emph{Faithfulness}. This operates one level down, at the level of the
objects themselves rather than their mere existence: it concerns whether
derivations and strategies correspond, not merely whether they coexist.
\end{itemize}
 
What we state and
prove in this section is adequacy: Theorem~\ref{thm:game-soundness} shows
that derivability in a base yields winning strategies in the corresponding
arena, Theorem~\ref{thm:strategy-derivability} establishes the converse, and
Corollary~\ref{cor:adequacy} records the biconditional. The proofs, however,
deliver more than the statements record: both are compositional
translations, under which plays track the construction of proofs in a base.
We return to this point in Remark~\ref{rem:faithfulness}.
 
The correspondence rests on a simple observation: the data of a base is
exactly the data of an arena. An atomic rule, read game-theoretically, is a
tree: its conclusion is a node at which the rule may be invoked, its
premisses are the challenges to which invoking it exposes the invoker, and
its discharged hypotheses are the concessions available for meeting those
challenges. Consider, for example, the rule \emph{from `Amber is a fox' and
`Amber is female', infer `Amber is a vixen'}, which determines the arena
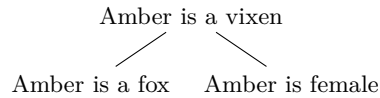
\begin{center} \small
    \begin{forest}
        [Amber is a vixen[Amber is a fox][Amber is female]]
    \end{forest}
\end{center}
relating the three logic-free sentences. A dialogue in this arena proceeds as
an argument would: Opponent opens by asking whether Amber is a vixen and
granting whatever assumptions are in play --- say, that Amber is a fox and
that Amber is female; Proponent invokes the rule by playing a question at the
root; Opponent challenges the premisses by questioning the child nodes;
Proponent answers each challenge, as she may, since Opponent has conceded the
corresponding facts; and, every premiss being met, Opponent concedes the
root, whereupon Proponent answers the initial question and the argument is
won. A strategy that wins in this way against every line of challenge is
nothing other than a derivation in the base.
 
The comparison proceeds according
to the following dictionary, which the results of this section prove exact
at the level of consequence:
\begin{center}
\begin{tabular}{c|cc}
    \emph{level} & \emph{proof-theoretic concept} & \emph{game-theoretic concept} \\[1mm] \hline
    data & base & arena \\
    dynamics & proof-search & play \\
    consequence & derivability & winning strategy
\end{tabular}
\end{center}
 
We begin by showing that bases and arenas are
isomorphic as data structures. The isomorphism is witnessed by a pair of
translations, defined by structural recursion, that convert atomic trees to
base rules and back.
\begin{definition}[Translations]
\label{def:arena-rule-function}
Define $\tobase{-}$ on rooted trees of atoms and $\toarena{-}$ on base rules by mutual
structural recursion:
\[
  \tobase{(\atree{}{p})} := \langle p \rangle \qquad   \tobase{(\atree{\Psi_1, \ldots, \Psi_n}{p})} :=
      \langle \{\tobase{\Psi_1}, \ldots, \tobase{\Psi_n}\}, p \rangle
\]
and
\[
  \toarena{\langle p \rangle} := (\atree{}{p}) \qquad \toarena{\langle \{\rho_1, \ldots, \rho_n\}, p \rangle} :=
      (\atree{\toarena{\rho_1}, \ldots, \toarena{\rho_n}}{p})
\]
where the $\Psi_i$ are atomic trees and the $\rho_i$ base rules. Extend both maps
elementwise: 
\begin{center}
\hspace{10mm}$\tobase{\arena{A}} := \{ \tobase{\Psi} \mid \Psi \in \arena{A} \}$ \quad and \quad
$\toarena{\base{B}} := \{ \toarena{\rho} \mid \rho \in \base{B} \}$ \fillBox
\end{center} 

\end{definition}
 
That the translations constitute an isomorphism is the content of the
following lemma, whose proof is deferred to
Appendix~\ref{app:sec-arenas-and-bases}.
\begin{lemma}
\label{lemma:inverse}
The functions $\tobase{-}$ and $\toarena{-}$ are mutually inverse: for every base $\base{B}$ and every arena $\arena{A}$,
$\toarena{\tobase{\arena{A}}} = \arena{A}$ and $\tobase{\toarena{\base{B}}} = \base{B}$.
\fillBox
\end{lemma}

The translations thus identify bases with arenas at the level of data. It
remains to show that the identification extends to the level of consequence:
derivability in a base obtains precisely when there is a winning strategy in
the corresponding arena. We treat the two directions separately, beginning
with the soundness of the game interpretation.
\begin{theorem}[Soundness for the game interpretation]
\label{thm:game-soundness}
    Let $\base{B}$ be a base, let $S$ be a finite set of atoms, and let $p$
    be an atom. If $S \deriv{\base{B}} p$, then
    $S \wstrat{\toarena{\base{B}}} p$.
\end{theorem}
\begin{proof}
By Definition~\ref{def:derivability-in-a-base}, the family of relations
$\deriv{\base{B}}$, indexed by bases $\base{B}$, is the least family closed
under (Ref) and (App). It therefore suffices to show that the family of
relations $\wstrat{\toarena{\base{B}}}$ is closed under the corresponding
conditions:
\[
\begin{array}{rl}
    \mbox{\rm (Ref)} & \mbox{if $p \in S$, then $S \wstrat{\toarena{\base{B}}} p$;} \\
    \mbox{\rm (App)} & \mbox{if $(Q_1 \brule q_1), \ldots, (Q_n \brule q_n)
        \brule r$ is a rule of $\base{B}$ and
        $S \wstrat{\toarena{\base{B} \cup Q_i}} q_i$,} \\
        & \mbox{for each $i = 1, \ldots, n$, then $S \wstrat{\toarena{\base{B}}} r$.}
\end{array}
\]
The theorem then follows by induction on the definition of
$S \deriv{\base{B}} p$.
 
\medskip
 
\noindent\emph{Case (Ref).} Suppose $p \in S$. By clauses~(i) and~(ii) of
Definition~\ref{def:play}, every play for $p$ from $S$ opens with the initial
question O?$p$ followed by the initial assumptions O!$s_1, \ldots,$ O!$s_k$,
where $s_1, \ldots, s_k$ enumerates $S$; since $p \in S$, the opening contains
the initial assumption O!$p$. Let $\sigma$ be the strategy that responds to
this opening with P!$p$. The move is legitimate: clause~(ix) is met by the
initial assumption O!$p$; clause~(x) permits answering the initial question,
again by that initial assumption; and clause~(viii) is satisfied vacuously,
the play containing no non-initial question. The resulting play
\[
\text{O?}p,\ \text{O!}s_1,\ \ldots,\ \text{O!}s_k,\ \text{P!}p
\]
is maximal --- the initial question has been answered, and by clause~(vii) no
question may be answered twice --- and it is the unique maximal play in which
Proponent follows $\sigma$. It contains a P-answer to the initial question,
so $\sigma$ is winning and $S \wstrat{\toarena{\base{B}}} p$.
 
\medskip
 
\noindent\emph{Case (App).} Suppose
$\rho = (Q_1 \brule q_1), \ldots, (Q_n \brule q_n) \brule r$ is a rule of
$\base{B}$ and, for each $i = 1, \ldots, n$, let $\sigma_i$ be a winning
strategy witnessing $S \wstrat{\toarena{\base{B} \cup Q_i}} q_i$. We
construct a winning strategy $\sigma$ witnessing
$S \wstrat{\toarena{\base{B}}} r$. We first record the shape of the arenas involved. The arena $\toarena{\{\rho\}}$ is a single
tree $\Psi$ whose root is labelled $r$, whose children are labelled
$q_1, \ldots, q_n$, and in which the subtrees below $q_i$ are exactly the
trees $\toarena{Q_i}$. Observe that
$\Psi$ is an element of $\toarena{\base{B}}$ and
$\toarena{\base{B} \cup Q_i} = \toarena{\base{B}} \cup \toarena{Q_i}$.
 
By clauses~(i) and~(ii) of Definition~\ref{def:play}, every play for $r$ from
$S$ opens O?$r$, O!$s_1$, \ldots, O!$s_k$, with $s_1, \ldots, s_k$
enumerating $S$. The strategy $\sigma$ extends this opening with the root
question P?$r$ at the root of $\Psi$, which clause~(iv) licenses: its node is
a root of the forest $\toarena{\base{B}}$, and its justifying question is the
initial question. Opponent now has two kinds of move available: by
clause~(v), Opponent may challenge a premiss, playing O?$q_i$ at a child of
the root of $\Psi$, justified by P?$r$; by clause~(vi), Opponent may instead
concede, answering P?$r$ with O!$r$. (If $n = 0$, conceding is Opponent's
only option.) These remain Opponent's only options throughout, alongside the
moves arising inside the sub-plays described below.
 
Suppose Opponent challenges with O?$q_i$. Then $\sigma$ continues by playing
according to $\sigma_i$, under the following identification: regard O?$q_i$
as the initial question, and the initial assumptions
O!$s_1, \ldots,$ O!$s_k$ of the ambient play as the initial assumptions, of a
simulated play in the arena $\toarena{\base{B} \cup Q_i}$. Moves of
$\sigma_i$ at nodes of $\toarena{\base{B}}$ carry over unchanged, since
$\toarena{\base{B}}$ is a sub-forest of both arenas. Moves of $\sigma_i$ at
nodes of $\toarena{Q_i}$ are transported to the corresponding nodes of $\Psi$
below $q_i$. The only clause of Definition~\ref{def:play} sensitive to this
relocation is the justification of questions at the roots of
$\toarena{Q_i}$: in the simulated play these are root questions, justified by
the initial question O?$q_i$ via clause~(iv), while in the ambient play the
corresponding nodes are children of $q_i$ in $\Psi$, so the same questions
are justified by O?$q_i$ via clause~(v). In either case the justifying
question is O?$q_i$, and every other justification is preserved. Clause~(x)
constrains only the answer to the ambient initial question, so it imposes
nothing on the sub-plays; clauses~(vi), (vii), and~(ix) hold in the ambient
play because they hold, move for move, in the simulated one.
 
Opponent may interleave the challenges O?$q_1, \ldots,$ O?$q_n$ and the
sub-plays they open. No interference results: justification in
Definition~\ref{def:play} relates \emph{occurrences} of moves, so each
non-initial move of the ambient play other than P?$r$ and O!$r$ descends, by
tracing justifying questions and the responses of $\sigma$, from a unique
challenge O?$q_i$, and $\sigma$ answers each O-move according to the
sub-strategy $\sigma_i$ of the challenge from which it descends. The
sub-plays therefore proceed independently, each a play of $\sigma_i$ under
the identification above. Since $\sigma_i$ is winning, every maximal
continuation of the sub-play opened by O?$q_i$ contains a P-answer to its
simulated initial question
(i.e., the move P!$q_i$)
and, since
$\sigma_i$ respects clause~(viii) in the simulated play, every question of
that sub-play has been answered by the time P!$q_i$ is played.
 
Finally, suppose Opponent concedes O!$r$ (whether immediately after P?$r$ or
after any number of challenges). Once every non-initial question of the play
has been answered, $\sigma$ answers the initial question with P!$r$. The move
is legitimate: clause~(ix) is met by O!$r$; clause~(x) is met because O!$r$
is an O-answer at a root vertex whose label coincides with that of the
initial question; and clause~(viii) holds by the hypothesis that no
non-initial question stands unanswered.
 
It remains to check that $\sigma$ is winning. Let $\pi$ be a maximal play in
which Proponent follows $\sigma$; recall that plays are finite sequences
(Definition~\ref{def:play}). The question P?$r$ must be answered in $\pi$:
while it stands unanswered, clause~(vi) provides Opponent with the legal move
O!$r$ whenever it is Opponent's turn, and responsiveness
(Definition~\ref{def:strategy}) ensures that the play cannot halt on
Proponent's turn; hence a play in which P?$r$ is unanswered is not maximal.
By the same reasoning, every challenge O?$q_i$ that Opponent raises is
answered with P!$q_i$, since $\sigma_i$ is winning, and every question
internal to a sub-play is answered before the corresponding P!$q_i$, as
observed above. Thus every non-initial question of $\pi$ is answered, O!$r$
occurs in $\pi$, and $\sigma$ concludes with P!$r$. Every maximal play in
which Proponent follows $\sigma$ therefore contains a P-answer to the initial
question, so $\sigma$ is winning and $S \wstrat{\toarena{\base{B}}} r$.
\end{proof}
\begin{theorem}[Completeness for the game interpretation]
\label{thm:strategy-derivability}
    Let $\arena{A}$ be an arena, let $P$ be a finite set of atoms, and let
    $r$ be an atom. If $P \wstrat{\arena{A}} r$, then
    $P \deriv{\tobase{\arena{A}}} r$.
\end{theorem}
\begin{proof}
 By
clauses~(i) and~(ii) of Definition~\ref{def:play}, every play opens with the initial question and the initial
assumptions
\[
\pi_0 := \text{O?}r,\ \text{O!}p_1,\ \ldots,\ \text{O!}p_k
\]
where $p_1, \ldots, p_k$ enumerates $P$.
Let $\sigma$ be a winning strategy witnessing $P \wstrat{\arena{A}} r$. Opponent is \emph{obstinate} in a play if, whenever he can ask a question not previously asked, he does so.
 Let $\ell$ be the minimal length of the maximal plays in which Opponent is obstinate and Proponent plays according to $\sigma$. We proceed by induction on $\ell$.
 
\medskip
 
\noindent\emph{Base case: $\ell = 1$.} In this case, there is only one move after the opening $\pi_0$. As $\sigma$ is winning, that move must be
P!$r$. By rule~(ix) of Definition~\ref{def:play}, a P-answer P!$r$ must be
preceded by an O-answer with the same label; since no move other than the
opening and P!$r$ occurs, that O-answer is one of the initial assumptions
O!$p_1, \ldots,$ O!$p_k$. Hence $r \in P$, and $P \deriv{\tobase{\arena{A}}} r$
by (Ref) of Definition~\ref{def:derivability-in-a-base}.
 
\medskip
 
\noindent\emph{Inductive step: $\ell > 1$.} Consider an arbitrary play $\pi$ for $r$ in $\arena{A}$ with assumptions $P$. By clause~(ix), a P-answer P!$r$ must be
preceded by an O-answer with the same label. If O!$r$ is in the initial segment, then $r \in P$ and $P \deriv{\tobase{\arena{A}}} r$
by (Ref). So suppose $r \notin P$.
By clause~(x), the P-answer
P!$r$ is licensed by an O-answer O!$r$ at a root vertex of $\arena{A}$. By
clause~(vi), that O-answer responds to an earlier P-question at the same
vertex; its vertex being a root of the forest, this is a root question. Thus $\pi$ has the form
\[
\pi_0,\ldots,\ \text{P?}r,\ \ldots,\ \text{O!}r,\ \text{P!}r,
\]
and there is a tree $\Psi \in \arena{A}$ whose root is labelled $r$ and at
which Proponent has played P?$r$:
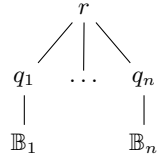
\begin{center} \small
    \begin{forest}
        [$r$[$q_1$[$\arena{B}_1$]][$\ldots$][$q_n$[$\arena{B}_n$]]]
    \end{forest}
\end{center}
Let $\rho := \tobase{\Psi}$; by Definition~\ref{def:arena-rule-function}, $\rho$ is a rule of $\tobase{\arena{A}}$.
We distinguish two cases.
 
\medskip
 
\noindent\emph{Case $n = 0$.} Then $\Psi$ is the single vertex labelled $r$,
so $\rho$ is the premiss-free rule $\brule r$. Using (App), with its premiss condition
vacuously met, yields $P \deriv{\tobase{\arena{A}}} r$.
 
\medskip
 
\noindent\emph{Case $n \ge 1$.} We claim that $P \wstrat{\arena{A} \cup \arena{B}_i} q_i$ for each $i$. To this end, we construct a winning strategy $\sigma_i$ witnessing this from $\sigma$.
Let
$\theta$ be any play for $q_i$ from $P$ in $\arena{A} \cup \arena{B}_i$, and
let $\theta'$ be obtained from $\theta$ by relocating every move at a node of
$\arena{B}_i$ to the corresponding node in the copy of $\arena{B}_i$ sitting
below $q_i$ in $\Psi$ (moves at nodes of $\arena{A}$ are left unchanged). We check that
\[
\pi_i := \pi_0,\ \text{P?}r,\ \text{O?}q_i,\ \theta'
\]
is a legal play for $r$ from $P$ in $\arena{A}$. The only clause of
Definition~\ref{def:play} sensitive to the relocation is the justification of
questions at the root $q_i$ of $\arena{B}_i$: in $\theta$ such a question is a
root question justified by the initial question O?$q_i$ (clause~(iv)), whereas
in $\pi_i$ the corresponding node is a child of $r$ in $\Psi$, so the
same question is justified by O?$q_i$ via clause~(v). In either case the
justifying question is O?$q_i$, and every other justification is preserved;
clauses~(vi)--(x) transfer move for move. Hence $\pi_i$ is a play for $r$ in $\arena{A}$ assuming $P$.
We now define $\sigma_i$ by relocating $\sigma$'s responses backwards: set
$\sigma_i(\theta)$ to be the sequence obtained from $\sigma(\pi_i)$ by
returning every move at a node of $\Psi$ below $q_i$ to the corresponding node
of $\arena{B}_i$ (and leaving moves at nodes of $\arena{A}$ unchanged). This is
well defined, since $\pi_i$ ends with an O-move exactly when $\theta$
does, and $\sigma$ is defined on such plays; coherence and responsiveness of
$\sigma_i$ follow from those of $\sigma$ under the relocation, which is a
bijection on the relevant nodes. Since $\sigma$ is a winning strategy for $r$ from $P$ in $\arena{A}$, every maximal play in which Proponent follows $\sigma_i$ contains the P-answer P!$q_i$; hence $\sigma_i$ is winning.
By construction, every maximal obstinate play for $q_i$ from $P$ in
$\arena{A} \cup \arena{B}_i$ in which Proponent follows $\sigma_i$ embeds as a
proper subsequence of a maximal obstinate play for $r$ from $P$ in $\arena{A}$
in which Proponent follows $\sigma$; hence its length is strictly smaller than
$\ell$, and the induction hypothesis yields $P \deriv{\tobase{\arena{A} \cup \arena{B}_i}} q_i$.

By the conversion between bases and arenas,
$\tobase{\arena{A} \cup \arena{B}_i}
     = \tobase{\arena{A}} \cup \tobase{\arena{B}_i}
     = \tobase{\arena{A}} \cup \{\, \tobase{\arena{B}_i} \brule q_i \,\}$.
Now $\tobase{\arena{B}_i} \brule q_i$ is precisely the $i$-th premiss pair of
the rule $\rho$, so writing $Q_i := \tobase{\arena{B}_i}$ for the rules
discharged by that premiss we have
$P \deriv{\tobase{\arena{A}} \cup Q_i} q_i$ for each $i = 1, \ldots, n$.
Since $\rho$ is a rule of $\tobase{\arena{A}}$, clause (App) of
Definition~\ref{def:derivability-in-a-base} yields
$P \deriv{\tobase{\arena{A}}} r$, as required.
\end{proof}
 
Combining the two theorems with Lemma~\ref{lemma:inverse}, we obtain the
adequacy of the game interpretation.
\begin{corollary}[Adequacy]
\label{cor:adequacy}
Let $S$ be a finite set of atoms and let $p$ be an atom. For every base $\base{B}$,
$S \deriv{\base{B}} p$ iff $S \wstrat{\toarena{\base{B}}} p$
and, equivalently, for every arena $\arena{A}$,
$S \wstrat{\arena{A}} p$ iff $S \deriv{\tobase{\arena{A}}} p$.
\end{corollary}
\begin{remark}[Faithfulness]
\label{rem:faithfulness}
Corollary~\ref{cor:adequacy} is extensional: it asserts the coincidence of
two relations, not a correspondence between their witnesses. The proofs,
however, are constructive and compositional. The strategy exhibited in the
proof of Theorem~\ref{thm:game-soundness} is defined by recursion on the
derivation: each application of (App) becomes a root question together with
sub-strategies meeting Opponent's challenges to the premisses of the rule
invoked, and each appeal to (Ref) becomes an immediate answer from the
initial assumptions. Conversely, the proof of
Theorem~\ref{thm:strategy-derivability} reads a derivation off a winning
strategy. In this sense, plays track the construction of proofs in a base.
Elevating this observation to a full-and-faithful correspondence ---
fullness: every winning strategy is the interpretation of some derivation;
faithfulness: distinct derivations, up to a suitable notion of identity, are
interpreted by distinct strategies --- would give an atomic analogue of the
full completeness theorems of game
semantics~\cite{AbramskyJagadeesan1994,HYLAND2000285}. Doing so requires
fixing notions of identity for derivations and for strategies, and we leave
it to future work; see Section~\ref{sec:discussion}. \fillBox
\end{remark}
 
Corollary~\ref{cor:adequacy} discharges the dictionary with which this
section opened: bases and arenas are the same structures and, at the level
of consequence, derivability and the existence of winning strategies
coincide. In particular, provability-in-a-base may be replaced by the
existence of a winning strategy in the corresponding arena without loss.
Section~\ref{sec:metatheory} makes exactly this replacement throughout Sandqvist's
clauses, yielding the game-extension semantics for $\IPL$.
\section{Logical Metatheory}
\label{sec:metatheory}
We have thus far established that, at the atomic level, the
proof-theoretic and the game-theoretic readings based on proof-search coincide: derivability in a
base is exactly the existence of a winning strategy in the corresponding arena
(Corollary~\ref{cor:adequacy}). The correspondence is, so far, limited to \emph{atomic} formulae. The
question of this section is whether the games extend from atoms to formulae:
What is it for Proponent to \emph{win} a formula $\phi$ in an arena $\arena{A}$, and
does the resulting notion of validity characterize $\IPL$?

There are some intuitive choices of definitions. For example, Proponent wins a conjunction $\phi \land \psi$ in an arena $\arena{A}$ just
in case she wins each conjunct $\phi$ and $\psi$ in $\arena{A}$. This captures the intensional characterization of a
conjunction as two separate obligations. Similarly, Proponent
wins an implication $\phi \to \psi$ in $\arena{A}$ just in case in any arena
$\arena{B}$ extending $\arena{A}$ in which she wins $\phi$, she also wins $\psi$. These heuristic accounts are unproblematic.
 
Disjunction requires more care, and it is here that the proof-search motivation
does real work. A na\"ive approach would be that Proponent wins the disjunction $\phi \lor \psi$ in $\arena{A}$
just in case she wins either $\phi$ or $\psi$ in $\arena{A}$.
This condition is sufficient: if Proponent wins a disjunct, then she wins the disjunction. But it is not necessary: that Proponent wins a
disjunction does not yield a winning strategy for either disjunct. The lottery example by Gheorghiu and Gu~\cite{Gheorghiu2026KnowledgeBases} makes this clear: in playing a lottery one knows that \emph{some} ticket will win, but does not know \emph{which} ticket will win. So what is the appropriate treatment for disjunction?
 
The
proof-search reading helps motivate an answer. A winning strategy
is a policy for \emph{continuing the search for a proof}; accordingly, to hold a
disjunction is not to hold a disjunct but to know how every search that could
be completed from $\phi$, and could be completed from $\psi$, may be
completed outright. That is, the disjunction is explicated by the
\emph{continuations} it serves: whatever (atomic) goal $p$ is winnable from
$\phi$ and winnable from $\psi$ is winnable simpliciter. This is a
continuation-passing-style reading of $\lor$ --- compare the interpretation
of classical disjunction via the $\lambda\mu\nu$-calculus and continuation
models~\cite{PYM2001315,Pym_ReductiveLogic}. This is precisely the game-theoretic
counterpart of Sandqvist's second-order clause. Indeed, the reading can be
made exact: Gheorghiu~\cite{Gheorghiu2026SupportIsSearch} shows that support
in a fixed base coincides with proof-search in a second-order hereditary
Harrop logic program under a CPS encoding of the connectives, with
$\phi \lor \psi$ encoded as
$\forall X\,((\phi \to X) \to (\psi \to X) \to X)$. The quantification over
extensions $\arena{B} \supseteq \arena{A}$ reflects that search may enlarge
the arena --- Proponent must be able to continue however the available rules
grow --- and the restriction to atomic continuations keeps the clause
pre-logical.
\begin{definition}[Support in an Arena]
\label{def:support_in_arena}
Let $\arena{A}$ be an arena. The \emph{support} relation
$\arenasupp{\arena{A}}$, relating finite (possibly empty) sets of formulae to
formulae, is defined inductively by the clauses of
Figure~\ref{fig:arena-support}, in which, in (Inf), $\Delta \neq \emptyset$. We say that
$\Gamma$ \emph{supports} $\phi$, written $\Gamma \arenasupp{} \phi$, iff
$\Gamma \arenasupp{\arena{A}} \phi$, for every arena $\arena{A}$. \fillBox
\end{definition}
\begin{figure}[t]
    \hrule
    \begin{align}
        &\arenasupp{\arena{A}} p
            &&\text{iff} \qquad\wstrat{\arena{A}} p
            \tag{At} \\
        &\arenasupp{\arena{A}} \phi \to \psi
            &&\text{iff} \qquad\phi \arenasupp{\arena{A}} \psi
            \tag{$\to$} \\
        &\arenasupp{\arena{A}} \phi \land \psi
            &&\text{iff} \qquad \arenasupp{\arena{A}} \phi \text{ and } \arenasupp{\arena{A}} \psi
            \tag{$\land$} \\
        &\arenasupp{\arena{A}} \bot
            &&\text{iff} \qquad\arenasupp{\arena{A}} p \text{, for every atom } p \in \At
            \tag{$\bot$} \\
        &\arenasupp{\arena{A}} \phi \lor \psi
            &&\text{iff} \qquad\text{for all } \arena{B} \supseteq \arena{A} \text{ and every atom } p \in \At \text{,} \notag \\
            & && \qquad \qquad\text{if } \phi \arenasupp{\arena{B}} p \text{ and } \psi \arenasupp{\arena{B}} p \text{, then } \arenasupp{\arena{B}} p
            \tag{$\lor$} \\
        &\Delta \arenasupp{\arena{A}} \phi
            &&\text{iff} \qquad \text{for all } \arena{B} \supseteq \arena{A} \text{, if } \arenasupp{\arena{B}} \psi \text{ for all } \psi \in \Delta \text{, then } \arenasupp{\arena{B}} \phi
            \tag{Inf}
    \end{align} 
    \caption{Support in an Arena}
    \hrule
    \label{fig:arena-support}
\end{figure}
 
Since the clauses of Figure~\ref{fig:arena-support} are those of
Figure~\ref{fig:sandqvist-ipl} with provability-in-a-base replaced, at the
base case, by the existence of a winning strategy, the results of
Section~\ref{sec:arenas-and-bases} transfer the two relations into one
another. We call this semantics  \emph{game-extension
semantics} (G-eS).
\begin{corollary} \label{cor:supp_equivalence}
Let $\Gamma$ be a finite set of formulae, let $\phi$ be a formula, let $\arena{A}$ be an arena, and let $\base{B}$ be a base. Then
\[
\mbox{$\Gamma \arenasupp{\arena{A}} \phi$ \quad iff \quad $\Gamma \supp{\tobase{\arena{A}}} \phi$,}
\]
and
\[
\mbox{$\Gamma \arenasupp{\toarena{\base{B}}} \phi$ \quad iff \quad $\Gamma \supp{\base{B}} \phi$.}
\]
\end{corollary}
\begin{proof}
    By induction on $\phi$. The base case is
    Corollary~\ref{cor:adequacy}; the inductive cases follow from the
    clause-by-clause match of Figure~\ref{fig:arena-support} with
    Figure~\ref{fig:sandqvist-ipl}, together with Lemma~\ref{lemma:inverse},
    which carries the quantification over extensions across the translations.
\end{proof}
 
Combining Corollary~\ref{cor:supp_equivalence} with
Theorem~\ref{thm:sandqvist-soundness-completeness} yields the main logical
result:
\begin{corollary} \label{cor:NJ_GeS_equivalence}
    Let $\Gamma$ be finite. Then $\Gamma \vdash_{\NJ} \phi$ iff
    $\Gamma \arenasupp{} \phi$. \fillBox
\end{corollary}
 
Although Corollary~\ref{cor:NJ_GeS_equivalence} follows by transfer, both
soundness and completeness admit direct proofs, obtained by adapting
Sandqvist's arguments~\cite{Sandqvist_Tor} to the game-theoretic setting;
these are given in Appendix~\ref{app:metatheory} and remain constructive and
elementary. Soundness is standard: each rule of NJ preserves support in an
arena. Completeness proceeds by constructing a \emph{simulation arena}
$\sarena$ in which each subformula $\xi$ of the formulae in question is
represented by a fresh atom $\flatset{\xi}$, and support and proof-search are
simulated simultaneously. It is worth pausing on the trees that $\sarena$
allots to a disjunction, for they exhibit the continuation-passing reading
concretely: alongside the trees encoding the introductions,
$\atree{\flatset{\phi}}{\flatset{(\phi \lor \psi)}}$ and
$\atree{\flatset{\psi}}{\flatset{(\phi \lor \psi)}}$, the arena contains, for
each atomic continuation $\flatset{\chi}$, the tree
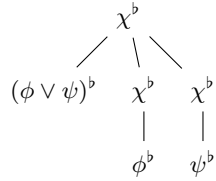
\begin{center}
\small
    \begin{forest}
        [$\chi^\flat$[$(\phi \lor \psi)^\flat$][$\chi^\flat$[$\phi^\flat$]][$\chi^\flat$[$\psi^\flat$]]]
    \end{forest}
\end{center}
 
Read game-theoretically: Proponent answers a question at $\flatset{\chi}$ by invoking the disjunction together with two sub-searches, one continuing from
$\flatset{\phi}$ and one from $\flatset{\psi}$. A winning strategy for a disjunction is thus used exactly as the clause ($\lor$) says it may be: not by producing a disjunct, but by supplying, for either eventuality, a continuation of the search \cite{GPRR-ContCompPtS2026}. That the completeness argument goes through with this clause --- and fails for the na\"ive one (cf.~\cite{Piecha2016completeness,
Piecha2019incompleteness,GPRR-ContCompPtS2026}) --- is the technical
expression of the same point.
 
We close by comparing this treatment of disjunction with that of the Pym-Ritter games from which our setup departs~\cite{Pym_ReductiveLogic}.
There, the arena for a formula containing $\psi_1 \lor \psi_2$ contains a switching vertex: Opponent's demand for the disjunction is met by Proponent playing $L$ or $R$, whereupon Opponent demands a proof of the chosen disjunct. This is the introduction rule read reductively, and it is apt when Proponent is \emph{constructing} a proof of the disjunction; but it exacts a commitment that proof-search cannot in general afford --- choose the wrong disjunct and the search fails although the goal is provable. In the Pym-Ritter framework, the remedy is a matter of \emph{control}: backtracking, modelled by embedding the intuitionistic games in classical ones, where Proponent may return to the switch and play the other move. Our clause internalizes that remedy semantically. Rather than committing to a disjunct and reserving the right to backtrack, a winning strategy for $\phi \lor \psi$ carries continuations for both disjuncts, so no commitment is ever made that the search might need to retract. Where the switching reading gives disjunction the type of a choice, the present reading gives it the type of its eliminations --- exactly the shift effected by continuation-passing-style translations, under which the control operators of classical search become intuitionistically definable~\cite{Pym_ReductiveLogic}. Sandqvist's clause
for $\lor$, which within B-eS can appear a technical device required for completeness, thus acquires a direct computational content in G-eS: it is the type of a strategy that knows how to continue the search, however the disjunction is eventually resolved \cite{GPRR-ContCompPtS2026}.

\section{Discussion} 
\label{sec:discussion}

The main contribution of this work is  the establishment of a direct correspondence between game-theoretic and proof-theoretic notions. Lemma \ref{lemma:inverse} identifies bases with arenas at the level of data and Corollary \ref{cor:adequacy} identifies, at the level of consequence, winning strategies in arenas and derivability in bases. Corollary \ref{cor:supp_equivalence} equates validity in bases with validity in arenas and Corollary \ref{cor:NJ_GeS_equivalence} demonstrates that the semantics developed in this paper is, in fact, a semantics for $\rm IPL$. In this sense, the semantics presented in this paper follows the inferentialist project: we have shown that concepts of game semantics can be reconstructed from the primitives of P-tS. 

Gheorghiu \cite{Gheorghiu2026SupportIsSearch},  showed that support in a  base coincides with uniform proof-search in a second-order hereditary Harrop program. As a corollary of the correspondences established here, this  result acquires an operational reading: validity in an arena is not merely characterized by winning strategies, but is executable, with strategies computed by goal-directed search against the arena's base.
     
The focus of this work is the development of a game semantics rooted in proof-theoretical concepts. Future directions might include an extension of the semantics presented in this paper to the cases corresponding to first- and higher-order predicate logic (cf.~\cite{Gheorghiu2025-FirstOrder,GP2025-P-tS-S-oL}). 
Connections with categorical accounts of B-eS (e.g., \cite{PymRitterRobinson2024}) and proof-relevant semantics should also be considered.  

One might also explore different notions of play and how they give rise to different logics. Specifically, how can substructural reasoning (cf. \cite{GheorghiuGuPym2023IMLL,gheorghiu_pts_bi_2025}) be captured by constraining the definition of play: contraction, for example, can be restricted by limiting the number of times players can move on the same node. Also, relevance can be imposed by a condition that forces the existence of a question relative to every vertex in an arena. 

One might also consider how  classical reasoning might be captured using the inferentialist game semantics apparatus. At minimum, classical logic can be obtained from the current game-extension semantics if the depth of trees in arenas is restricted to trees with maximum depth two and the disjunction clause is omitted (cf.~\cite{Gheorghiu2025-FirstOrder,Sandqvist2005inferentialist}). This paper offers an alternative way following \cite{Pym_ReductiveLogic}: to relax clause (ix) of Definition~\ref{def:play} so that for any P-answer $m_i$ there exists an O-question $m_k$ and an O-answer $m_j$ such that $m_i$ is hereditarily justified by $m_k$ and $m_j$ is an O-answer with the same label as $m_k$ and $k<j<i$. Another approach would be to develop concurrent games, as in \cite{alcolei_et_al:LIPIcs.CSL.2018.5}, to capture classical reasoning.   
    




\bibliography{references}

\appendix

\section{Proofs for Section~\ref{sec:games}}
\label{app:games}

\begin{proof}[Proof of Lemma~\ref{lem:monotonicity} (Monotonicity)]
Suppose $Q \wstrat{\arena{A}} r$ and $\arena{B} \supseteq \arena{A}$. There are
two cases.

\medskip

\noindent\emph{Case $r \in Q$.} Since $r \in Q$, the atom $r$ is among the
initial assumptions of any play for $r$ from $Q$. Hence Proponent may answer
the initial question immediately, and
\[
\text{O?}r,\ \text{O!}q_1,\ \ldots,\ \text{O!}q_n,\ \text{P!}r
\qquad (q_1, \ldots, q_n \in Q)
\]
is a winning strategy in any arena; in particular, $Q \wstrat{\arena{B}} r$.

\medskip

\noindent\emph{Case $r \notin Q$.} By clause~(x) of
Definition~\ref{def:play}, there is a tree $\Psi \in \arena{A}$ whose root
vertex is labelled $r$; let $\Psi^1, \ldots, \Psi^n$ be its immediate
subtrees:
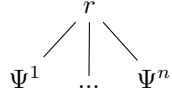
\begin{center}
    \begin{forest}
        [$r$[$\Psi^1$][...][$\Psi^n$]]
    \end{forest}
\end{center}
By (winning), a winning strategy witnessing
$Q \wstrat{\arena{A}} r$ begins with the initial question O?$r$ and the
initial assumptions O!$q_1$, \ldots, O!$q_n$ for $q_1, \ldots, q_n \in Q$. For
Proponent to answer the initial question, there must be an O-answer at the
root of $\Psi$ and hence, by clause~(vi) of Definition~\ref{def:play}, a
P-question at that node. Since $\arena{B} \supseteq \arena{A}$, we have
$\Psi \in \arena{B}$; and by Definition~\ref{def:atomic-arena} the root of
$\Psi$ is a P-node in $\arena{B}$, so Proponent may question it in
$\arena{B}$ exactly as in $\arena{A}$, whatever further trees $\arena{B}$
contains. Hence the strategy that answers the initial question O?$r$ from the
initial assumptions O!$q_1$, \ldots, O!$q_n$ in $\arena{A}$ answers it from
the same assumptions in $\arena{B}$; that is, $Q \wstrat{\arena{B}} r$.
\end{proof}

\section{Proofs for Section~\ref{sec:arenas-and-bases}}
\label{app:sec-arenas-and-bases}

\begin{proof}[Proof of Lemma~\ref{lemma:inverse}]
We prove two claims: first, that $\toarena{\tobase{\Psi}} = \Psi$ for every
atomic tree $\Psi$, by induction on the depth of $\Psi$; second, that
$\tobase{\toarena{\Gamma}} = \Gamma$ for every base rule $\Gamma$, by
induction on the level of $\Gamma$.

\medskip

\noindent\emph{First claim.}

\smallskip

\noindent\emph{Base case.} Suppose $\Psi = (\atree{q_1, \ldots, q_n}{p})$. By
Definition~\ref{def:arena-rule-function},
$\tobase{\Psi} = (q_1, \ldots, q_n \brule p)$, and, again by
Definition~\ref{def:arena-rule-function},
$\toarena{(q_1, \ldots, q_n \brule p)} = (\atree{q_1, \ldots, q_n}{p}) = \Psi$.

\smallskip

\noindent\emph{Inductive step.} Suppose
$\Psi = (\atree{\Theta_1, \ldots, \Theta_n}{p})$, where each $\Theta_i$ (with
$1 \leq i \leq n$) is an immediate subtree of $p$, and assume
$\toarena{\tobase{\Theta_i}} = \Theta_i$ for every $\Theta_i$. By
Definition~\ref{def:arena-rule-function}, applied twice,
\[
\begin{array}{rcl}
\toarena{\tobase{(\atree{\Theta_1, \ldots, \Theta_n}{p})}}
    & = & \toarena{(\tobase{\Theta_1}, \ldots, \tobase{\Theta_n} \brule p)} \\
    & = & (\atree{\toarena{\tobase{\Theta_1}}, \ldots , 
          \toarena{\tobase{\Theta_n}}}{p})
\end{array}
\]
By the induction hypothesis, the right-hand side is
$(\atree{\Theta_1, \ldots, \Theta_n}{p}) = \Psi$.

\medskip

\noindent\emph{Second claim.}

\smallskip

\noindent\emph{Base case.} Suppose $\Gamma = (q_1, \ldots, q_n \brule p)$. By
Definition~\ref{def:arena-rule-function},
$\toarena{\Gamma} = (\atree{q_1, \ldots, q_n}{p})$, and, again by
Definition~\ref{def:arena-rule-function},
$\tobase{(\atree{q_1, \ldots, q_n}{p})} = (q_1, \ldots, q_n \brule p) = \Gamma$.

\smallskip

\noindent\emph{Inductive step.} Suppose
$\Gamma = (\Delta_1, \ldots, \Delta_n \brule p)$, where each $\Delta_i$ (with
$1 \leq i \leq n$) is a base rule of lower level, and assume
$\tobase{\toarena{\Delta_i}} = \Delta_i$ for every $\Delta_i$. By
Definition~\ref{def:arena-rule-function}, applied twice,
\[
\begin{array}{rcl}
\tobase{\toarena{(\Delta_1, \ldots, \Delta_n \brule p)}}
    & = & \tobase{(\atree{\toarena{\Delta_1}, \ldots, \toarena{\Delta_n}}{p})} \\ 
    & = & (\tobase{\toarena{\Delta_1}}, \ldots , \tobase{\toarena{\Delta_n}} \brule p) 
\end{array}
\]
By the induction hypothesis, the right-hand side is
$(\Delta_1, \ldots, \Delta_n \brule p) = \Gamma$.
\end{proof}

\section{Proofs for Section~\ref{sec:metatheory}}
\label{app:metatheory}


That $\IPL$ is sound with respect to the semantics is shown in an entirely standard way. One shows that support in an arena is preserved by the rules of the system NJ for IPL \cite{Gentzen1935,Prawitz1965}, as given in Figure~\ref{fig:nj}. 

\begin{figure}[ht]
\hrule 
\[
\begin{array}{c@{\quad\quad\quad\quad}c}
    & \infer[\bot{E}]{\Gamma \vdash \phi}{\Gamma \vdash \bot} \\ [4pt] 
\infer[\wedge{I}]{\Gamma \vdash \phi \wedge \psi}{\Gamma \vdash \phi & \Gamma \vdash \psi}
    &  \infer[\wedge{E}]{\Gamma \vdash \phi}{\Gamma \vdash \phi \wedge \psi} \qquad 
            \infer[\wedge{E}]{\Gamma \vdash \psi}{\Gamma \vdash \phi \wedge \psi} \\ [4pt] 
\infer[\vee{I}]{\Gamma \vdash \phi \vee \psi}{\Gamma \vdash \phi} \qquad 
            \infer[\vee{I}]{\Gamma \vdash \phi \vee \psi}{\Gamma \vdash \psi}     
    &  \infer[\vee{E}]{\Gamma \vdash \chi}
            {\Gamma \vdash \phi \vee \psi & \Gamma, \phi \vdash \chi &
              \Gamma, \psi \vdash \chi} \\ [4pt]
\infer[\supset{I}]{\Gamma \vdash \phi \supset \psi}{\Gamma , \phi \vdash \psi}
    & \infer[\supset{E}]{\Gamma \vdash \psi}{\Gamma \vdash \phi & \Gamma \vdash \phi \supset \psi} \\ [4pt] 
\end{array}
\]
\caption{Natural Deduction System $\NJ$ for IPL (eliding $\top$)}
\vspace{1mm}
\hrule 
\label{fig:nj}
\end{figure}

\begin{theorem}[Soundness] \label{thm:soundness}
    If $\Gamma \vdash_{\rm NJ} \phi$, 
    then $\Gamma \arenasupp{} \phi$. 
\end{theorem}

The more interesting direction is completeness. It too is 
articulated through NJ. 
The proof of completeness is similar to that for the corresponding result in by Sandqvist \cite{Sandqvist_Tor}. We construct a `simulation' arena $\sarena$ by `flattening', $(-)^\flat$, which is defined in the proof of completeness, below. In this arena, we simultaneous simulate support and proof-search and show that they coincide. 

\begin{definition}[Simulation arena] \label{def:sim-arena}
We define the simulation arena $\sarena$ associated with a set of formulae $\Theta$. Let $\Xi$ be the set containing all elements  of $\Theta$ and their subformulae. Then associate with each element $\xi$ of $\Xi$ a distinct atomic counterpart $\xi^\flat$ such that

\begin{itemize}
\item $(\phi \land \psi)^\flat$:
\begin{center}
    \begin{forest}
        [\;$(\phi \land \psi)^\flat$
        [$\phi^\flat$]
        [$\psi^\flat$]]
    \end{forest}
        \hspace{0.5cm}
    \begin{forest}
        [$\phi^\flat$[$(\phi \land \psi)^\flat$]]
        \end{forest}
        \hspace{0.5cm}
        \begin{forest}
        [$\psi^\flat$[$(\phi \land \psi)^\flat$]]
    \end{forest}
\end{center}

\item $\phi \to \psi$

\begin{center}
    \begin{forest}   
        [$(\phi \to \psi)^\flat$
        [$\psi^\flat$[$\phi^\flat$]]]    
    \end{forest}
        \hspace{0.5cm}
    \begin{forest} 
        [$\psi^\flat$[$(\phi \to \psi)^\flat$][$\phi^\flat$]]
    \end{forest}
\end{center}

    \item $\phi \lor \psi$

\begin{center}
    \begin{forest}   
        [$\chi^\flat$[$(\phi \lor \psi)^\flat$][$\chi^\flat$[$\phi^\flat$]][$\chi^\flat$[$\psi^\flat$]]]
    \end{forest}
        \hspace{0.5cm}
    \begin{forest}
        [$(\phi \lor \psi)^\flat$[$\phi^\flat$]]
    \end{forest}
        \hspace{0.5cm}
    \begin{forest}
        [$(\phi \lor \psi)^\flat$[$\psi^\flat$]]
    \end{forest}
\end{center}

    \item $\bot$

\begin{center}
    \begin{forest}   
        [$\phi^\flat$[$\bot^\flat$]]
    \end{forest}
\end{center}
\end{itemize} 
This defines the simulation arena $\sarena$ associated with $\Theta$. \fillBox
\end{definition}
Atoms in $\sarena$ simulate the behaviour of the subformulae in $\phi$.

\begin{theorem}[Completeness] \label{thm:completeness} 
If $\Xi \arenasupp{} \zeta$, then $\Xi \vdash_{\rm NJ} \zeta$. 
\end{theorem}
\begin{proof}
Assume $\Xi \arenasupp{} \zeta$. Let $\Gamma$ be the set containing all members of $\Xi \cup \{\zeta\}$ and their subsentences. With every non-atomic $\gamma \in \Gamma$ associate an atom $\gamma^\flat \notin \Gamma$ in such a way that $\gamma_{1}^\flat \neq \gamma^\flat_2$ whenever $\gamma_1 \neq \gamma_2$. Also for every atom $g \in \Gamma$, set $g^\flat = g$. Conversely, with every atom $p$ associate a sentence $\shrp{p}$ in such a way that $\shrp{(\gamma^\flat)} = \gamma$ for every $\gamma  \in \Gamma$; when $p$ in not in the range of $\flat$, set $\shrp{p} = p$. For any $\flatset{\Phi} = {\phi^\flat | \phi \in \Phi}$ and $\shrpset{P} = {\shrp{p} | p \in P}$.

To prove completeness, Sandqvist, In \cite{Sandqvist_Tor}, constructs a simulation base $\sbase$ (Figure \ref{fig:sbase}) with the following properties:

\begin{figure}[ht]
\hrule 
\[
\begin{array}{c@{\quad\quad\quad\quad}c}
    & \infer[\bot{E}]{\phi^\flat}{\bot^\flat} \\ [4pt] 
\infer[\wedge{I}]{(\phi \wedge \psi)^\flat}{\phi^\flat & \psi^\flat}
    &  \infer[\wedge{E}]{\phi^\flat}{(\phi \wedge \psi)^\flat} \qquad 
            \infer[\wedge{E}]{\psi^\flat}{(\phi \wedge \psi)^\flat} \\ [4pt] 
\infer[\vee{I}]{(\phi \vee \psi)^\flat}{\phi^\flat} \qquad 
            \infer[\vee{I}]{(\phi \vee \psi)^\flat}{\psi^\flat}     
    &  \infer[\vee{E}]{\chi^\flat}
            {(\phi \vee \psi)^\flat & \deduce{\chi^\flat}{[\phi^\flat]}    &
            \deduce{\chi^\flat}{[\psi^\flat]}} \\ [4pt]
\infer[\supset{I}]{(\phi \supset \psi)^\flat}
    {\deduce{\psi^\flat}{[\phi^\flat]}}
    & \infer[\supset{E}]{\psi^\flat}{\phi^\flat & (\phi \supset \psi)^\flat} \\ [4pt] 
\end{array}
\]
\caption{Sandqvist's simulation base $\sbase$ \cite{Sandqvist_Tor}}
\vspace{1mm}
\hrule 
\label{fig:sbase}
\end{figure}
\begin{itemize}
    \item[($\dagger$)] for every $\gamma \in \Gamma$ and every $\base{B} \supseteq \sbase$, $\supp{\base{B}}\gamma^\flat$ iff $\supp{\base{B}} \gamma$
    \item[($\ddagger$)] for any $P$ and $q$, if $P \supp{\sbase} q$, then $\shrp{P} \vdash_{\rm NJ} \shrp{q}$ 
\end{itemize}
 From these two properties, Sandqvist gets to the desired conclusion by the following argument: from the hypothesis that $\Xi \Vdash \zeta$ it follows that $\Xi^\flat \supp{\sbase} \zeta^\flat$ --- for if $\base{B \supseteq \sbase}$ and $\supp{\base{B}} \xi^\flat$ for every $\xi^\flat \in \Xi^\flat$, then by $(\dagger)$, $\supp{\base{B}} \xi$ for every $\xi \in \Xi$, whence $\supp{\base{B}} \zeta$ since $\Xi \Vdash \zeta$, whence $\supp{\base{B}} \zeta^\flat$ by $(\dagger)$ again; but,  then, because $T \supp{\base{B}} q$ iff $T \supp{\base{B}} q$ (Theorem 3.1 of \cite{Sandqvist_Tor}), it follows that $\Xi^\flat \supp{\sbase} \zeta^\flat$, whence by $(\ddagger)$, $\shrp{(\Xi^\flat)} \supp{\sbase} \shrp{(\zeta^\flat)}$, which is just to say that $\Xi \vdash_{\rm NJ} \zeta$, as desired.

Notice that $\tobase{\sarena}$ is exactly Sandqvist's simulation base $\sbase$ and, conversly, $\toarena{\sbase} = \sarena$. From this fact, toghether with corollary \ref{cor:adequacy} it follows that:
\[
S \supp{\sbase} p \quad \text{iff} \quad S \wstrat{\toarena{\sbase}} p,
\]
and, equivalently,
\[
S \wstrat{\sarena} p \quad \text{iff} \quad S \supp{\tobase{\sarena}} p.
\]
which is just to say,
\[
S \wstrat{\sarena} p \quad \text{iff} \quad S \supp{\sbase} p.
\]
Also, for every extension $\arena{B \supseteq N}$ if follows that $S \wstrat{\arena{B}} p$ iff $S \supp{\tobase{\arena{B}}} p$ and, conversely for every extension $\base{B \supseteq \sbase}$, $S \supp{\base{B}} p$ iff $S \wstrat{\toarena{\base{B}}} p$. Hence, by corollary \ref{cor:adequacy} together with Theorem \ref{cor:supp_equivalence} it follows that:
\begin{itemize}
    \item[($\dagger*$)] for every $\gamma \in \Gamma$ and every $\arena{B} \supseteq \sarena$, $\supp{\arena{B}}'\gamma^\flat$ iff $\supp{\arena{B}}' \gamma$
    \item[($\ddagger*$)] for any $P$ and $q$, if $P \wstrat{\sarena} q$, then $\shrp{P} \vdash_{\rm NJ} \shrp{q}$ 
\end{itemize}
The same argument as in \cite{Sandqvist_Tor} can then be used to conclude that $\Xi^\flat \wstrat{\sbase} \zeta^\flat$ and, consequently, by $(\ddagger *)$, $\shrp{(\Xi^\flat)} \vdash_{\rm NJ} \shrp{(\zeta^\flat)}$, which is just to say $\Xi \vdash_{\rm NJ}  \zeta$.
\end{proof}

\end{document}